\documentclass{amsart}
\usepackage{amssymb,amsmath,amsthm}
\title{Predicative toposes}
\author{Benno van den Berg}
\usepackage{ast2}
\usepackage[margin=1.4in]{geometry}
\usepackage{proof}

\def\epi{\ensuremath{\twoheadrightarrow}}

\begin{document}

\maketitle

\begin{abstract}
We explain the motivation for looking for a predicative analogue of the notion of a topos and propose two definitions. For both notions of a predicative topos we will present the basic results, providing the groundwork for future work in this area.
\end{abstract}

\section{Introduction}

Topos theory is a highly successful chapter in the categorical analysis of constructive logic. Originally toposes were invented by Grothendieck with the aim of proving the Weil conjectures in algebraic geometry. Later it became clear that these Grothendieck toposes have a rich internal logic, which led Lawvere and Tierney to formulate the notion of an ``elementary topos''. For quite some time, the only examples of such elementary toposes that people knew about were Grothendieck toposes and the free topos. This changed when Martin Hyland discovered the effective topos at the beginning of the eighties. More recently, people have discovered Dialectica toposes whose internal logic is related to various functional interpretations. As a result, we have a common framework in which to study topological, sheaf and Kripke models (because these are special cases of Grothendieck toposes) as well as realizability and functional interpretations.

But the power and expressiveness of the internal logic of a topos also creates a certain \emph{embarras de richesse}. Indeed, it is far stronger than what most constructive mathematicians are willing to use in their work. Nowadays, most constructivists point to systems like Martin-L\"of's Type Theory (MLTT) \cite{martinlof84} or Peter Aczel's constructive set theory {\bf CZF} \cite{aczelrathjen01} as providing the kind of system in which they want their work to be formalisable. These systems are much weaker, however, than the internal language of a topos, which is a form of full higher-order arithmetic (if we use topos to mean elementary topos with natural numbers object, as we will do in this paper).

The key difference is that MLTT and {\bf CZF} are systems which are (generalised) predicative; this means that, although they may accept a wide variety of inductively defined sets, they do not include the powerset axiom. In contrast, the internal language of a topos is impredicative, due to the presence of a subobject classifier and power objects more generally.

This has led people to wonder whether there could be a kind of ``predicative topos'': a topos-like structure whose internal language would be closer to the formal systems constructivists actually use. Ideally, one could develop for these predicative toposes an equally rich and interesting theory, which would be closer in spirit to constructive mathematics as it is practised.

The first proposals for such a notion of a predicative topos were put forward by Moerdijk and Palmgren (see their notions of a $\Pi W$-pretopos \cite{moerdijkpalmgren00} and a stratified pseudotopos \cite{moerdijkpalmgren02}). But before we go into this, it is good to first get an idea of what properties one would like predicative toposes to have. A list of desiderata would probably include:
\begin{enumerate}
\item A predicative topos should at least be a locally cartesian closed pretopos with natural numbers object.
\item Every topos should be a predicative topos.
\item The setoids in MLTT should form a predicative topos.
\item The sets in {\bf CZF} should form a predicative topos.
\item Predicative toposes should be environment in which one can do formal topology.
\item Predicative toposes should be closed under internal sheaves and realizability.
\item There should be interesting examples of predicative toposes that are not toposes.
\end{enumerate}
It might be good to make a few remarks concerning these requirements.
\begin{description}
\item[Ad 3] Because MLTT is an intensional theory, the ``sets'' one works with are not really the types, but types together with an equivalence relations (i.e., setoids). This is related to the distinction that Bishop makes between sets and presets. So we try to capture the categorical properties of the sets rather than the presets (see \cite{palmgren12} for a discussion of this point).
\item[Ad 5] Recently a lot of constructivists have worked on formal topology. Essentially, this is locale theory in a predicative metatheory. The idea behind point (5) is that predicative toposes should stand to formal spaces in the same way as toposes stand to locales. That means both that one should be able to do formal topology inside predicative toposes and that one should be able to take sheaves over a formal space internally to a predicative topos and again get a predicative topos.
\item[Ad 6] In the same way, we want to be able to take sheaves or do realizability inside a predicative topos and again get a predicative topos back. As a result, sheaves and realizability should not only be model constructions giving rise to predicative toposes, but actually be closure properties of the class of predicative toposes.
\end{description}
It turns out that it is very hard, indeed impossible, to satisfy all of these desiderata.

First of all, there are two dilemmas. The first dilemma is how many inductively defined sets one adds to one's notion of a predicative topos. It is often said that MLTT is an open framework to which one may add as many inductively generated sets as one wants. But we need to fix something in order to get going. In {\bf CZF}, however, there are very few inductively generated sets. In this connection, Peter Aczel proposed to add another axiom to {\bf CZF} (the Regular Extension Axiom, or REA). The question is to which extent we follow him in this.

The second dilemma is how much choice we add to our notion of a predicative topos. Most constructivists are willing to countenance certain forms of choice; and, as we will see, one seems to be compelled to use some amount of choice if one can no longer reason impredicatively. In a general topos, however, only very few choice principles hold. 

These dilemmas turn into a real problem if one starts to do formal topology. The general notion of a formal space is really too ill-behaved: in a predicative setting only those formal spaces that are set-presented (in Peter Aczel's terminology) behave well. For example, working in {\bf CZF}, one need not get a model of {\bf CZF} if one takes sheaves over a formal space that is not set-presented (see \cite{grayson83, gambino06}).

On the other hand many formal spaces cannot be shown to be set-presented inside {\bf CZF}; this includes formal Baire space (see \cite{bergmoerdijk10c}). Formal Baire space is an example of an ``inductively generated formal space'' \cite{coquandetal03} and to show that these are set-presented we have to go beyond {\bf CZF}. So if we insist that internally to a predicative topos we should be able to prove that ``inductively generated formal spaces are set-presented'', we have to drop requirement (4); and, indeed, that is what we will do.

As it turns out, proving the statement ``inductively generated formal spaces are set-presented'' seems to be one of those things for which we need to add some inductively generated sets and choice. Let us take the inductively generated sets first.

What kind of inductively defined sets would one like to add? The obvious choice seems to be the W-types of MLTT. In \cite{moerdijkpalmgren00}, Moerdijk and Palmgren observe that these can be captured categorically as initial algebras for polynomial functors, and adding these to a locally cartesian closed pretopos leads to their notion of a $\Pi W$-pretopos.

Unfortunately, we still seem to need some choice. We have not been able to prove that
\begin{itemize}
\item[(i)] inductively generated formal spaces are set-presented,
\item[(ii)] taking sheaves over a set-presented formal space gives you a category which again has W-types 
\end{itemize}
without some form of choice (see \cite{bergmoerdijk10b} in connection with (ii)). We do not have proofs that these results cannot be proved internally to a $\Pi W$-pretopos (that would probably be very hard), but we are sceptical that it can be done.

So what form of choice do we add? If we would follow the lead of MLTT again, we would postulate the existence of enough projectives (because that is what happens in the category of setoids). That would allow us to prove the two theorems above, but it would not be compatible with desideratum (6): the existence of enough projectives is not stable under sheaves.

For solving precisely this issue, Moerdijk and Palmgren introduced in \cite{moerdijkpalmgren02} an axiom they called the ``Axiom of Multiple Choice'' (AMC). We have renamed their axiom the ``Strong Axiom of Multiple Choice'' (SAMC), because there is a weaker axiom which does the job as well: we have decided to reserve the name ``Axiom of Multiple Choice'' for this weaker axiom (see \cite{bergmoerdijk10}). 

As a result we have the following notion of a predicative topos: a $\Pi W$-pretopos satisfying AMC. It turns out that in a predicative topos both (i) and (ii) are provable. In fact, these results do not only hold for formal spaces but for general sites; we will work this out in Section 8. (The results will look a bit different there: in fact, without a distinction between sets and classes as in {\bf CZF}, it is hard to talk about formal spaces that are not set-presented: set-presentedness will become part of the definition.)

$\Pi W$-pretoposes that satisfy SAMC will be called strong predicative toposes. We have no clear preference as to what should be the ``true'' notion of a predicative topos: whether it should be the strong or the weak one. Right now, everything one wants to do with SAMC one can also do with ordinary AMC; but SAMC holds in all the examples we know of and its stability under various constructions may be easier to show due to its equivalence to a new axiom called RP (see Section 5).

But one may object to both notions, because it is not clear that they satisfy requirement (2), saying that every topos should be an example; and, in fact, they do not, for AMC is not provable in higher-order arithmetic. It is not even provable in {\bf ZF}, as we will show in Section 6 (note that Rathjen and Lubarsky have shown the same for REA, see \cite{rathjenlubarsky03}). This might dismay topos theorists, but it should be kept in mind that:
\begin{itemize}
\item AMC follows from the axiom of choice, so is true in $\Sets$. And because AMC will be stable under realizability and sheaves, it will hold in realizability and sheaf categories defined over $\Sets$.
\item AMC could have certain desirable consequences, which are not provable in higher-order arithmetic. For example, in {\bf ZF} it implies that every algebraic theory has free algebras, a result which is not provable in {\bf ZF} proper (see Section 6).
\end{itemize}

However, both our notions of a predicative topos satisfy all the other criteria in the list (all of them, apart from (2) and (4), that is). Essentially that is the upshot of \cite{moerdijkpalmgren02,bergmoerdijk08, bergmoerdijk09, bergmoerdijk10b, bergmoerdijk10}. But there the results are stated in the context of algebraic set theory and Aczel's constructive set theory {\bf CZF}, which might not be congenial to everyone. Indeed, one of the main purposes of this paper is to explain these results in a more categorical language and make clear what is their relevance to predicative topos theory. We also collect important results from other sources hoping that this paper can act as a starting point and reference for future work in this area.

But we do not only collect or restate old results; there will be some novel results as well. Concretely, these are:
\begin{enumerate}
\item We show that SAMC is equivalent to a new axiom, which we call RP. 
\item We use this to show that SAMC is preserved by ex/reg-completion and realizability.
\item We show that AMC is unprovable in ZF.
\item We show that the subcountable objects in the effective topos and the ex/lex-completions of the categories of topological spaces and $T_0$-spaces are examples of strong predicative toposes.
\item We show that (i) is provable in predicative toposes for general sites.
\end{enumerate}

The contents of this paper will be as follows. In Section 2 we will establish categorical notation and terminology and give a precise definition of the notion of a $\Pi W$-pretopos. In Section 3 we will introduce both variants of the Axiom of Multiple Choice and in Section 4 we will introduce a reflection principle and show it is equivalent to SAMC. We will study AMC in the context of the classical set theory {\bf ZF} in Section 5. In Section 6 we will give examples of predicative toposes and we will also discuss closure under realizability. Section 7 will be devoted to closure under sheaves and showing that (i) and (ii) hold in predicative toposes for general sites. Finally, to conclude our paper, we will indicate directions for future research in Section 8.

\section{Categorical preliminaries}

It will be the purpose of this section to establish the categorical terminology and notation that we will need. In particular, we define precisely the notion of a $\Pi W$-pretopos. But first we recall some results from the theory of exact completions; these will become important because we will obtain most interesting examples of predicative toposes as exact completions.

\begin{defi}{regularcat}
Let \ct{C} be a category. A map $f: B \to A$ is a \emph{cover}, if the only subobject of $A$ through which it factors is $A$ itself. The category \ct{C} will be called \emph{regular}, if it has finite limits, every map factors as a cover followed by a mono, and covers are stable under pullback. A functor $F: \ct{C} \to \ct{D}$ will be called \emph{regular}, if it preserves finite limits and covers.
\end{defi}

\begin{defi}{eqrelexactcat}
Let \ct{C} be a category with finite limits. A subobject $(r_0, r_1): R \subseteq X \times X$ in \ct{C} will be called an \emph{equivalence relation}, if for any object $A$ in \ct{C} the induced map
\[ {\rm Hom}(A, R) \to {\rm Hom}(A, X)^2: f \mapsto (r_0 \circ f, r_1 \circ f) \]
is injective and determines an equivalence relation on ${\rm Hom}(A, X)$. A map $q: X \to Q$ will be called a \emph{quotient} of the equivalence relation, if the diagram
\diag{
  R \ar@<.8ex>[r]^{r_0} \ar@<-.8ex>[r]_-{r_1} & X \ar[r]^q & Q
}
is both a pullback and a coequaliser, in which case the diagram is called \emph{exact}. It is called \emph{stably exact}, when for any \func{p}{P}{Q} the diagram
\diag{
  p^{*}R \ar@<.8ex>[r]^{p^{*}r_0} \ar@<-.8ex>[r]_-{p^{*}r_1} & p^{*}X \ar[r]^{p^{*}q} & P
}
is also exact. The category \ct{C} will be called \emph{exact}, if it is regular and every equivalence relation fits into a stably exact diagram.
\end{defi}

Every category with finite limits can be turned into an exact category in a suitably universal way. More precisely, for any category with finite limits \ct{C} there is an exact category $\ct{C}_{ex/lex}$ (its ex/lex-completion) together with a finite limit preserving functor ${\bf y}: \ct{C} \to \ct{C}_{ex/lex}$ such that for any exact category \ct{D} precomposing with {\bf y} induces an equivalence of functor categories between the finite limit preserving functors from \ct{C} to \ct{D} and the regular functors from $\ct{C}_{ex/lex}$ to \ct{D}. In addition, regular categories can be turned into exact categories, while preserving the regular structure: so for any regular category \ct{C} there is an exact category $\ct{C}_{ex/reg}$ (its ex/reg-completion) together with a regular functor ${\bf y}: \ct{C} \to \ct{C}_{ex/lex}$ such that for any exact category \ct{D} precomposing with {\bf y} induces an equivalence of functor categories between the regular functors from \ct{C} to \ct{D} and the regular functors from $\ct{C}_{ex/reg}$ to \ct{D}. For the purposes of this paper, we do not need to know how one constructs these completions (for that see \cite{carboni95}), but we do need to be able to recognise them.

\begin{defi}{coveringfunctor}
If $X$ and $Y$ are two objects in a category \ct{C}, we say that $Y$ \emph{is covered by} $X$ if there is a cover $q: X \to Y$. A functor $F: \ct{C} \to \ct{D}$ will be called \emph{covering}, if every object $D$ in \ct{D} is covered by one in the image of $F$. It will be called \emph{full on subobjects} if every subobject of an object in the image of $F$ is isomorphic to an object in the image of $F$.
\end{defi}

\begin{defi}{extproj}
Let \ct{C} be a regular category. An object $P$ in \ct{C} is a \emph{projective} if every cover \func{p}{X}{P} splits. We will say that the category \ct{C} has \emph{enough projectives} if every object in \ct{C} is covered by a projective.
\end{defi}

\begin{prop}{recogexlexcompl}
Let \ct{C} be a category with finite limits, \ct{D} be an exact category and $F: \ct{C} \to \ct{D}$ be a finite limit preserving functor. Then \ct{D} is equivalent to the ex/lex-completion of \ct{C} via an equivalence which commutes with $F: \ct{C} \to \ct{D}$ and ${\bf y}: \ct{C} \to \ct{C}_{ex/lex}$ iff
\begin{enumerate}
\item $F$ is full and faithful,
\item $F$ is covering,
\item and the objects in the image of $F$ are, up to isomorphism, the projectives in \ct{D}.
\end{enumerate}
\end{prop}
\begin{proof}
See \cite{carboni95}.
\end{proof}

\begin{prop}{recogexregcompl}
Let \ct{C} be a regular category, \ct{D} be an exact category and $F: \ct{C} \to \ct{D}$ be a regular functor. Then \ct{D} is equivalent to the ex/reg-completion of \ct{C} via an equivalence which commutes with $F: \ct{C} \to \ct{D}$ and ${\bf y}: \ct{C} \to \ct{C}_{ex/reg}$ iff $F$ is full and faithful, covering, and
full on subobjects.
\end{prop}
\begin{proof}
See \cite{carboni95}.
\end{proof}

\begin{coro}{exregexlexcoincidence}
Let \ct{C} be a regular category and assume that \ct{C} has enough projectives and the projectives in \ct{C} are closed under finite limits (meaning that the limit of any finite diagram whose objects are all projective is again projective). Then ${\rm Proj}(\ct{C})_{ex/lex}$ and $\ct{C}_{ex/reg}$ are equivalent, where ${\rm Proj}(\ct{C})$ is the full subcategory of \ct{C} on the projectives.
\end{coro}
\begin{proof}
It is clear that the inclusion $F: {\rm Proj}(\ct{C}) \to \ct{C}_{ex/reg}$ is full and faithful, cartesian and covering. So we need to show that the objects in the image of $F$ are, up to isomorphism, the projectives of $\ct{C}_{ex/reg}$.

So suppose $P$ is projective in $\ct{C}_{ex/reg}$. Since ${\bf y}: \ct{C} \to \ct{C}_{ex/reg}$ is covering, there exists a cover ${\bf y}X \to P$. Because $P$ is projective, this cover splits and $P$ must actually be subobject of ${\bf y}X$. Since {\bf y} is full on subobjects, $P$ is isomorphic to some ${\bf y}Q$ with $Q$ in $\ct{C}$. This object $Q$ has to be projective in \ct{C}, because ${\bf y}$ is regular and full and faithful; so $P$ is isomorphic to an object in ${\rm Proj}(\ct{C})$.

Conversely, suppose $Q$ is projective in $\ct{C}$; we want to show that ${\bf y}Q$ is projective in $\ct{C}_{ex/reg}$. So let $q: X \to {\bf y}Q$ be a cover. Since {\bf y} is covering, we can cover $X$ via some $p: {\bf y}A \to X$. Since {\bf y} is full, there is a map $f: A \to Q$ in \ct{C} with ${\bf y}f = qp$. It is not hard to see that $f$ is a cover and therefore must have a section $s: Q \to A$. Now $qp{\bf y}s = {\bf y}f{\bf y}s = \id$, so $q$ splits.
\end{proof}

We will now define the notion of a $\Pi W$-pretopos. Before we can do that, we first need to define locally cartesian closed categories, polynomial functors and W-types.

\begin{defi}{lccc}
A category \ct{C} which has all finite limits is called \emph{locally cartesian closed} if for any map $f: Y \to X$ the pullback functor
\[ f^*: \ct{C}/X \to \ct{C}/Y \]
has a right adjoint. This right adjoint is usually denoted $\Pi_f$.
\end{defi}

Alternatively we could say that a category is locally cartesian closed if it has all finite limits and all its slice categories are cartesian closed.

\begin{defi}{Wtypes}
Let $f: B \to A$ be a map in a locally cartesian closed category \ct{C}. The \emph{polynomial functor} associated to $f$ is the endofunctor
\diag{ P_f: \ct{C} \ar[r]^{B \times -} & \ct{C}/B \ar[r]^{\Pi_f} & \ct{C}/A \ar[r]^{\Sigma_A} & \ct{C}, }
where $\Sigma_A$ is the forgetful functor $\ct{C}/A \to \ct{C}$ which only remembers the domain.

The initial algebra for $P_f$, whenever it exists, is called the \emph{W-type} associated to $f$. If all polynomial functors have initial algebras, we say that \ct{C} \emph{has W-types}.
\end{defi}

For more on these W-types, we refer to \cite{moerdijkpalmgren00} and \cite{berg09}.

\begin{defi}{lextpretopos}
A category which has all finite limits and finite sums which are disjoint and stable under pullback is called \emph{lextensive}. A category which is both lextensive and exact is called a \emph{pretopos}.
\end{defi}

\begin{defi}{PiWpretopos}
A $\Pi W$-pretopos is a locally cartesian closed pretopos having all W-types.
\end{defi}

When we write topos in this paper, we will always mean an elementary topos with natural numbers object.

\begin{theo}{topisPiWpretopos}
Every topos is a $\Pi W$-pretopos.
\end{theo}
\begin{proof}
It is well-known that every topos is a locally cartesian closed pretopos. The existence of W-types is Proposition 2.3.5 in \cite{pareschumacher78} (see also \cite{blass83} and \cite{moerdijkpalmgren00}).
\end{proof}

\section{The Axiom of Multiple Choice}

In this section we will introduce both variants of the Axiom of Multiple Choice (see also \cite{moerdijkpalmgren02, bergmoerdijk10}). Before we do that, we first define covering and (strong) collection squares.

Throughout this section we will work in a category \ct{E} which is locally cartesian closed, lextensive and regular. Such a category is in particular a Heyting category, which means that it has an internal logic which is a typed version of first-order intuitionistic logic with (dependent) product and sum types. We will often exploit this fact. Readers who need to know more about the internal logic we refer to \cite{makkaireyes77, johnstone02b}.

\begin{defi}{coveringsquare}
We call a square
\diag{ D \ar[r]^{q} \ar[d]_g & B \ar[d]^f \\
C \ar[r]_{p} & A }
a \emph{covering square}, if both $p$ and the canonical map $D \to B \times_A C$ are covers. When such a covering square exists with a map $g$ on the left and $f$ on the right, we say that $g$ \emph{covers} $f$ and $f$ \emph{is covered by} $g$.
\end{defi}

\begin{rema}{asymmetry} Note that the orientation is important for being a covering square, as the maps $p$ and $f$ play different roles. So being a covering square is really a property of an \emph{oriented} square. To fix things, we will always draw covering squares in such a way that the property holds from ``left to right'' (as in the definition), instead of from ``top to bottom''. For (strong) collection squares the same convention will apply.
\end{rema}

\begin{defi}{collectionsquare}
A square as the one above will be called a \emph{collection square}, if the following statement holds in the internal logic:  for every $c \in C$ and cover $e: E \epi D_c$ there is a $c' \in C$ with $p(c) = p(c')$ and a map $h: D_{c'} \to D_c$ over $B$ which factors through $e$. It will be called a \emph{strong collection square} if $h$ can be chosen to be a cover.
\end{defi}

\begin{rema}{different} In \cite{bergmoerdijk10} we required collection squares to be covering as well. This turned out to be less convenient here, so we no longer make that part of the definition. But in case the square \emph{is} covering the definition here agrees with the one in \cite{bergmoerdijk10}, as the following lemma shows.
\end{rema}

\begin{lemm}{equivcharcollsq}
A covering square
\diag{ D \ar[r]^{q} \ar[d]_g & B \ar[d]^f \\
C \ar[r]_{p} & A }
is a collection square iff the following statement holds in the internal logic: for all $a \in A$ and covers $e: E \epi B_a$
there is a $c \in p^{-1}(a)$ and a map $t: D_c \to E$ such that the triangle
\diag{ & E \ar[dr]^e & \\
D_c \ar[ur]^t \ar[rr]_{q_c = q \upharpoonright D_c} & & B_a }
commutes. 
\end{lemm}
\begin{proof}
An easy exercise in using the internal logic: suppose the square is a collection square and we are given a cover $e: E \to B_a$. First, we find a $c \in C$ with $p(c) = a$, because $p$ is a cover, and we pull back $e$ along $q_c$ as in:
\diag{ Z \ar@{->>}[d]_f \ar@{->>}[r]^r & E \ar@{->>}[d]^e \\
D_c \ar@{->>}[r]_{q_c} & B_ a. }
Since the square is collection, we find a $c' \in C$ with $p(c') = p(c) = a$ and a map $s: D_{c'} \to Z$ such that $q_cfs = q_{c'}$. Hence $c'$ and $t := rs$ are as desired.

Conversely, suppose we are given an element
$c \in C$ and a cover $e: E \epi D_c$. Put $a := p(c)$ and consider $q_ce: E \to B_a$. We find an element $c' \in C$ with $p(c') = a$ and a function $s: D_{c'} \to E$ such that $q_{c'} = q_ces$. Then $c'$ and $h := es$ are as desired.
\end{proof}

We can now state the
\begin{description}
\item[(Strong) Axiom of Multiple Choice] Every map $f: Y \to X$ fits into the right-hand side of a square which is both covering and (strong) collection.
\end{description}
We will abbreviate the Axiom of Multiple Choice as AMC, and its strong version as SAMC. Clearly, SAMC implies AMC. (We believe that the converse fails. We do not have a proof, however, and we expect that finding one will be very difficult.)

The Axiom of Multiple Choice is really a weak choice principle. As such, it is implied by related choice principles. We will discuss a few of them, following \cite{moerdijkpalmgren02}.

\begin{defi}{intproj}
An object $P$ is called a \emph{choice object}, when for any cover $Y \to X$ and any arrow $T \times P \to X$, there exists a cover $T' \to T$ and map $T' \times P \to Y$ such that the square
\diag{ T' \times P \ar[r] \ar@{->>}[d] & Y \ar@{->>}[d] \\
       T \times P \ar[r] & X }
commutes.  An arrow $f: Y \to X$ is called a \emph{choice map}, if it is a choice object in the slice category $\ct{E}/X$.
\end{defi}

Since we are assuming that our base category \ct{E} is locally cartesian closed, an object $P$ in \ct{E} is a choice object iff the functor $(-)^P$ preserves covers (see e.g.~\cite{maclanemoerdijk92}). This is the same as saying that the following scheme is valid in the internal logic of \ct{E} for any object $X$:
\[ (\forall p \in P) \, (\exists x \in X) \, \varphi(p,x) \rightarrow (\exists f \in X^P) \,  (\forall p \in P) \, \varphi(p, f(p)). \]
For this reason choice objects are often called  \emph{internal projectives}, but we will not use this terminology here.

\begin{prop}{choicemapsandSAMC}
If every map $f: B \to A$ in \ct{E} is covered by a choice map, meaning that it fits into a covering square
\diag{ D \ar[r]^{q} \ar[d]_g & B \ar[d]^f \\
C \ar[r]_{p} & A }
with a choice map $g$ on the left, then \ct{E} satisfies SAMC.
\end{prop}
\begin{proof}
This is immediate using the internal logic: since $g$ is a choice map, every cover $q: E \to D_c$ has a section (internally).
\end{proof}

\begin{prop}{enoughprojandSAMC}
If \ct{E} has enough projectives and the projectives in \ct{E} are closed under pullbacks, then every map in \ct{E} is covered by a choice map. Hence \ct{E} satisfies SAMC.
\end{prop}
\begin{proof}
If \ct{E} has enough projectives, then every map $f: B \to A$ is covered by a map between projectives $g: D \to C$: we choose $C$ to be a projective cover of $A$ and $D$ to be a projective cover of $C \times_A B$. It is immediate that $g$ is projective in $\ct{E}/C$. But it is also a choice object in this category, because if  projectives are closed under products, then every projective is automatically a choice object (that is easy to see); this applies to $\ct{E}/C$, because we are assuming the projectives in \ct{E} to be closed under pullbacks and we know that $C$ is projective.
\end{proof}

\section{A reflection principle}

In this section we introduce a reflection principle reminiscent of Peter Aczel's Regular Extension Axiom. It will turn out that this axiom is equivalent to SAMC. Here, as in the previous section, we will work in a locally cartesian closed lextensive and regular category \ct{E}.

\begin{defi}{smallmaps}
By a \emph{class of small maps} \smallmap{S} in \ct{E} we will mean a class of maps in \ct{E} satisfying at least the following axioms:
\begin{itemize}
\item[(S1)] (Pullback stability) If the square
\diag{ D \ar[d]_g \ar[r] & B \ar[d]^f \\
C \ar[r] & A }
is a pullback and $f \in \smallmap{S}$, then also $g \in \smallmap{S}$.
\item[(S2)] (Covering) If a square as the one above is covering and $g \in \smallmap{S}$, then also $f \in \smallmap{S}$.
\item[(S3)] (Collection) Any two arrows $p: Y \to X$ and $f: X \to A$ where $p$ is a cover and $f$ belongs to \smallmap{S} fit into a covering square
\diag{ Z \ar[d]_g \ar[r] & Y \ar@{->>}[r]^p & X \ar[d]^f \\
B \ar@{->>}[rr]_h & & A,}
where $g$ belongs to \smallmap{S}.
\end{itemize}
A \emph{representation} for a class of small maps \smallmap{S} is a morphism $\pi: E \to U \in \smallmap{S}$ such that any morphism $f \in S$ is covered by a pullback of $\pi$. More explicitly: any $f: Y \to X  \in \smallmap{S}$ fits into a diagram of the form
\diag{Y \ar[d]_f & A \ar[d] \ar[r] \ar@{->>}[l] & E \ar[d]^{\pi} \\
X & B \ar[r] \ar@{->>}[l] & U, }
where the left hand square is covering and the right hand square is a pullback. The class \smallmap{S} will be called \emph{representable}, if it has a representation.
\end{defi}

\begin{defi}{RPaxiom}
The \emph{reflection principle RP} states that every map belongs to a representable class of small maps.
\end{defi}

\begin{theo}{SAMCRPequivalent}
RP and SAMC are equivalent.
\end{theo}
\begin{proof}
We first show that RP implies SAMC. Suppose $f: B \to A$ belongs to a class of small maps \smallmap{S} with representation $\pi: E \to U$. Form a covering square
\diag{ D \ar[r] \ar[d] & B \ar[d]^f \\
C \ar[r] & A }
in which
\begin{eqnarray*}
C & = & \{ (a, u, p) \, : \, a \in A, u \in U, p: E_{u} \epi B_a \}, \\
D & = & \{ (a, u, p, e) \, : \, (a, u, p) \in C, e \in E_u \},
\end{eqnarray*}
and the maps are the obvious projections. We will now show that it is a strong collection square as well.

Reason in the internal logic: suppose $c = (a, u, p) \in C$ and $t: Y \epi E_u \cong D_c$ is a cover. Using the collection axiom we find a small $S$ and a cover $q: S \epi E_u$ factoring through $t$. By representability, there is a cover of the form $r: E_{u'} \epi S$ with $u' \in U$. Putting $h := qr$ and $c' = (a, u', ph)$, we see that $h$ is really a map $D_{c'} \cong E_{u'} \epi E_{u} \cong D_c$ over $B$; in addition, it factors through $t$ because $q$ does.

In order to show that SAMC implies RP, let $f: Y \to X$ be a map fitting into a diagram
\diag{ D \ar[d]_{\rho} \ar@{->>}[r] &  A \times_X Y \ar[d]^{q^*f} \ar@{->>}[r] & Y \ar[d]^f
\\
C \ar@{->>}[r] & A \ar@{->>}[r]_q & X}
with the square on the left both covering and strong collection, while the one on the right is a pullback with an cover $q$ at the bottom. Let $\smallmap{S}$ be the class of maps $g: T \to S$ fitting into diagrams of the form
\diag{T \ar[d]_g & \bullet \ar[d] \ar[r] \ar@{->>}[l] & D \ar[d]^{\rho} \\
S & \bullet \ar[r] \ar@{->>}[l] & C, }
where the left hand square is covering and the right hand square is a pullback. Alternatively, we can say that $g: T \to S$ belongs to \smallmap{S} iff the following statement holds in the internal logic:
\[ (\forall s \in S) \, (\exists c \in C) \, (\exists p: D_c \to T_s) \, p \mbox{ is surjective}. \]
It is immediate that \smallmap{S} is a class of maps satisfying (S1) and (S2). It also satisfies the collection axiom (S3); to show this we reason internally. Let $q: E \epi Z$ be a cover and $Z$ be ``small'' (in the sense of \smallmap{S}), meaning that there is a $c \in C$ and a cover $p: D_c \to Z$. Take the following pullback:
\diag{ F \ar@{->>}[d]_r \ar@{->>}[r] & E \ar@{->>}[d]^q \\
D_c \ar@{->>}[r]_p & Z }
Using the strong collection square property, we find a $c' \in C$ and a cover $h: D_{c'} \epi D_c$ factoring through $r$. Since $D_{c'}$ is obviously small, we have shown that \smallmap{S} satisfies (S3).

So \smallmap{S} is a class of small maps. As it is representable by construction and contains $f$, the proof is finished.
\end{proof}

\begin{rema}{RPandREA}
Although RP is reminiscent of REA, this result probably means that RP is a bit stronger. For although SAMC implies REA (see \cite{moerdijkpalmgren02}), we would conjecture that the converse fails.
\end{rema}

\section{The Axiom of Multiple Choice in ZF}

In this section we study AMC in the context of the classical set theory {\bf ZF}. Set-theoretically, AMC is the following statement:
\begin{quote}
For every set $X$ there is a set of surjections $\{ p_i: Y_i \to X \}$ onto $X$ such that for every surjection $q: Z \to X$ there is an $i \in I$ and a function $f: Y_i \to Z$ such that $p_i = q \circ f$.
\end{quote}

The following result was already shown for strong AMC by Rathjen (see \cite{rathjen06b}).

\begin{theo}{AMCimpliesclassregcard} In {\bf ZF}, AMC implies that arbitrarily large regular cardinals exist.
\end{theo}
\begin{proof} Let $\kappa$ be a cardinal and use AMC to find a family $(p_i: B_i \to \kappa \, | \, i \in I)$ of surjections onto $\kappa$ such that for any surjection onto $\kappa$ we find one in this family which factors through it. Put $A = I + \{ 0, 1 \}$, $B_0 = \emptyset$, $B_1 = \{ 0 \}$ and let
\[ f: \sum_{a \in A} B_a \to A \]
be the obvious projection. In addition, write $W = W(f)$ for the W-type associated to $f$. 

By transfinite induction on $W$, we can now define a map $m: W \to Ord$ as follows:
\begin{eqnarray*}
m({\rm sup}_i(t)) & = & {\rm sup}\{ m(tb) \, : \, b \in B_i \}, \\
m({\rm sup}_0(t)) & = & 0, \\
m({\rm sup}_1(t)) & = & m(t(0)) + 1.
\end{eqnarray*}
Observe that:
\begin{enumerate}
\item 0 lies in the image on $m$.
\item The image of $m$ is closed under taking successor ordinals.
\item The image of $m$ is closed under suprema of cardinality $\kappa$: for if $\beta = {\rm sup}(\alpha_\lambda)$ and $\alpha_\lambda \in m(W)$ for every $\lambda \in \kappa$, then
\[ (\forall \lambda \in \kappa) \, (\exists w \in W) \, m(w) = \alpha_\lambda. \]
By choice of  $(f_i: B_i \to \kappa \, | \, i \in I)$ there now is an $i \in I$ and a function $t: B_i \to W$ such that
\[ (\forall b \in B_i) \, m(tb) = \alpha_{p_i(b)}. \]
Hence $\beta = m({\rm sup}_i(t))$.
\end{enumerate}
So if $\alpha = {\rm sup}(m(W))$, then $\alpha$ is a limit ordinal of cofinality strictly greater than $\kappa$. But since the cofinality of a limit ordinal is always a regular cardinal, we have found a regular cardinal whose size is strictly greater than $\kappa$.
\end{proof}

\begin{coro}{AMCunprovableinZF} AMC is unprovable in
{\bf ZF}.
\end{coro}
\begin{proof}
This follows from \reftheo{AMCimpliesclassregcard} and a celebrated result of Gitik \cite{gitik80}, saying {\bf ZF} is consistent with the statement that all uncountable cardinals are singular.
\end{proof}

\begin{rema}{ongitiksproof}
To construct his model of {\bf ZF} in which all uncountable cardinals are singular, Gitik uses the existence of a proper class of strongly compact cardinals. It is known that large cardinal assumptions are necessary for that result; whether they are also necessary for refuting AMC is an open problem.
\end{rema}

\begin{coro}{AMCfreealginZF}
In {\bf ZF}, AMC implies that every algebraic theory has free algebras.
\end{coro}
\begin{proof}
Follows from \reftheo{AMCimpliesclassregcard} above and Proposition 2 on page 151 of \cite{blass83}.
\end{proof}

That every algebraic theory has free algebras is known not to be provable in {\bf ZF} proper (see \cite{blass83}; of course, it is provable in {\bf ZFC}). So this is an example of an interesting mathematical statement which can be proved using AMC, but is not provable in {\bf ZF}. It would be interesting to see more statements of this kind. 

\section{Examples of predicative toposes}

The previous section concludes our discussion of the Axiom of Multiple Choice. We will now turn to the theory of predicative toposes. As we explained in the introduction, we will define these as follows: 

\begin{defi}{predtopos}
A \emph{(strong) predicative topos} is a $\Pi W$-pretopos satisfying the (strong) Axiom of Multiple Choice.
\end{defi}

In this section we will mainly collect some examples of predicative toposes. Of course, the most interesting are the ones that are not toposes.

\begin{theo}{setoidsexample}
The category of setoids in MLTT is a strong predicative topos.
\end{theo}
\begin{proof}
In \cite[Section 7]{moerdijkpalmgren00} it is shown that the setoids form a $\Pi W$-pretopos. Since in this category every map is covered by a choice map (see \cite[Lemma 12.3]{moerdijkpalmgren02}), it is a strong predicative topos by \refprop{choicemapsandSAMC}.
\end{proof}

As we said in Section 2, most examples are built using the theory of exact completions. 

\begin{theo}{predtopasexactlexcompl}
The ex/lex-completion of a $\Pi W$-pretopos is a strong predicative topos.
\end{theo}
\begin{proof}
This follows from Theorem 4.7 in \cite{berg09} in combination with \refprop{recogexlexcompl} and \refprop{enoughprojandSAMC} above.
\end{proof}

\begin{rema}{exactcomplclosproperty}
The previous theorem implies in particular that ex/lex-completion is a closure property of predicative toposes. In this respect they differ from toposes: an ex/lex-completion of a topos is a topos only if it has a generic proof (see \cite{menni03}). Examples of toposes without a generic proof can be found in \cite{menni03, menni07}; their exact completions are then examples of predicative toposes that are not toposes.
\end{rema}

For ex/reg-completions we have the following very useful result:
\begin{theo}{exregcomplforpredtoposes}
If \ct{E} is a locally cartesian closed, lextensive and regular category which has all W-types and satisfies (strong) AMC, then its ex/reg-completion is a (strong) predicative topos.
\end{theo}
\begin{proof}
For ordinary AMC, this result was proved in the context of algebraic set theory in \cite{bergmoerdijk08, bergmoerdijk10}. So it remains to show that SAMC is preserved by ex/reg-completion. We could prove this directly, but it can also be shown more elegantly using the equivalence of SAMC with RP.

If $f$ is a map in the ex/reg-completion, it is covered by one of the form ${\bf y}g$, where ${\bf y}$ is the embedding of \ct{E} in $\ct{E}_{ex/reg}$. Let \smallmap{S} be a class of small maps containing $g$ in \ct{E}. Then \smallmap{S} determines a class of small maps $\overline{\smallmap{S}}$ in $\ct{E}_{ex/reg}$, whose elements are precisely those morphisms covered by ones of the form ${\bf y}h$ with $h$ belonging to $\smallmap{S}$ (see \cite[Lemma 5.7]{bergmoerdijk08}). So $f \in \overline{\smallmap{S}}$ and $f$ belongs to a class of small maps.
\end{proof}

\begin{theo}{realforpredtoposes}
(Strong) predicative toposes are closed under internal realizability.
\end{theo}
\begin{proof}
For ordinary predicative toposes this was shown in the context of algebraic set theory in \cite{bergmoerdijk09, bergmoerdijk10}. So it remains to show that SAMC in preserved by realizability. We sketch an argument, relying on the theory developed in \cite{bergmoerdijk09} and the equivalence of SAMC with RP.

Let $f$ be a map in the realizability category of \ct{E}. Since this realizability category is the ex/reg-completion of the category of assemblies, $f$ will be covered by a map $g$ between assemblies. Choose a class of small maps $\smallmap{S}$ in \ct{E} such that $g$ becomes a display map between assemblies, and let $\overline{\smallmap{S}}$ be the class of maps in the realizability category that are covered by the display maps. Then $\overline{\smallmap{S}}$ is a class of small maps containing $f$.
\end{proof}

Recall that an object $X$ is called \emph{subcountable} if it is a quotient of a subobject of $\NN$.

\begin{theo}{subcineffaspredtop}
The subcountable objects in the effective topos form a strong predicative topos.
\end{theo}
\begin{proof}
Recall that the subcountable objects coincide with the discrete ones and that the discrete objects that are also $\lnot\lnot$-separated are called the modest sets (see \cite{hylandrobinsonrosolini90} and \cite[Section 3.2.6]{vanoosten08}). Since every discrete object is covered by a $\lnot\lnot$-separated subobject of $\NN$ (due to Shanin's Principle) and $\lnot\lnot$-separated objects are closed under subobjects, the discrete objects are the ex/reg-completion of the modest sets. So if suffices to prove that the modest sets satisfy all the hypotheses of \reftheo{exregcomplforpredtoposes}; for that, see \cite{bauer00}, with the validity of SAMC following from \refprop{enoughprojandSAMC} and the results on projective modest sets in \cite{bauer00}.
\end{proof}

Clearly, the subcountable objects cannot form a topos: the subcountability of ${\rm Pow}(\NN)$ leads to a contradiction by Cantor's diagonal argument.

\begin{theo}{topexlexaspredtop}
The ex/lex-completions of the category of topological spaces and the category of $T_0$-spaces are strong predicative toposes.
\end{theo}
\begin{proof}
To simplify the proof we first make a number of definitions (see \cite{bauerbirkedalscott04}): first of all, let $\ALat$ be the category of algebraic lattices. An \emph{assembly} over $\ALat$ is a triple $(X, \sigma, \alpha)$ where $X$ is a set, $\sigma$ is an algebraic lattice and $\alpha$ is a function which assigns to every $x \in X$ an inhabited subset of $\sigma$. A morphism $f: (X, \sigma, \alpha) \to (Y, \tau, \beta)$ is a function $f: X \to Y$ for which there is a morphism $r: \sigma \to \tau$ in $\ALat$ such that
\[ (\forall x \in X) \, (\forall n \in \alpha(x)) \, rn \in \beta(fx) \]
(we will say that $r$ \emph{tracks} $f$). The resulting 
category will be denoted by $\Asm[\ALat]$. An assembly $(X, \sigma, \alpha)$ will be called \emph{modest}, if $\alpha$ maps distinct elements to disjoint subsets of $\sigma$. We will write $\Mod[\ALat]$ for the full subcategory of $\Asm[\ALat]$ whose objects are the modest assemblies over $\ALat$ (this category is equivalent to the category of equilogical spaces). 

It is easy to see that in both categories the projective objects are those assemblies $(X, \sigma, \alpha)$ for which $\alpha(x)$ is always a singleton, the projectives are closed under finite limits and that every object is covered by such a projective. Since the full subcategories on the projectives are equivalent to the category of topological spaces and $T_0$-spaces respectively (see \cite{birkedaletal98}), it follows from \refcoro{exregexlexcoincidence} that the ex/lex-completion of the category of topological spaces is equivalent to the ex/reg-completion of $\Asm[\ALat]$, while the ex/lex-completion of the category of $T_0$-spaces is equivalent to the ex/reg-completion of $\Mod[\ALat]$; so it suffices to prove that both $\Asm[\ALat]$ and $\Mod[\ALat]$ satisfy the hypotheses of \reftheo{exregcomplforpredtoposes}.

Using that $\ALat$ is cartesian closed and has finite disjoint sums, it is easy to see that both $\Asm[\ALat]$ and $\Mod[\ALat]$ are categories which are locally cartesian closed, lextensive and regular. In addition, they satisfy SAMC because of \refprop{enoughprojandSAMC}.
So it remains to show that both categories have W-types. We will just outline the construction, leaving verifications to the reader. 

The key observation is that $\ALat$ is a category of domains in which one can prove that functors of the form
\[ F(X) = \rho \times (\sigma \to X) \]
have initial algebras. So if $f: (B, \sigma, \beta) \to (A, \rho, \alpha)$ is a morphism between assemblies, let $\mu$ be an initial algebra for $F$ and let
\[ s: \rho \times (\sigma \to \mu) \to \mu \]
be the algebra map. Now write $W = W(f)$ for the W-type associated to $f$ in $\Sets$, and define by transfinite recursion:
\[ \delta: W(f) \to \mu: {\rm sup}_a(t) \mapsto \{ \, s(x, y) \, : \, x \in \alpha(a), y: \sigma \to \mu \mbox{ tracks } t  \, \}. \]
If $V = \{ w \in W \, : \, \delta(w) \mbox{ is inhabited } \}$, then $(V, \mu, \delta)$ is the W-type associated to $f$. Moreover, if $(B, \sigma, \beta)$ and $(A, \rho, \alpha)$ are modest, then so is $(V, \mu, \delta)$.
\end{proof}

\begin{rema}{nottoposes}
Both exact completions were already known to be locally cartesian closed pretoposes (see \cite{birkedaletal98}). It also follows from known results that neither of the two predicative toposes can be a topos: for a proof that the ex/lex-completion of the category of topological spaces is not a topos, see \cite{lietzstreicher02}. The ex/lex-completion of the category of $T_0$-spaces cannot be a topos, because it is not well-powered: in fact, the category $\Mod[\ALat]$ already fails to be well-powered (see \cite{bauerbirkedalscott04}), so the same applies to its ex/reg-completion.
\end{rema}

\section{Sheaves for predicative toposes}

As we explained in the introduction, a crucial result in formal topology says that set-presented formal spaces can be constructed using inductive definitions. As predicative toposes should be an environment in which one can do formal topology, it will be important to show that one can perform this construction internally to a predicative topos. But we can do more: inductively generated formal spaces are a special case of Grothendieck sites constructed from general sites. And also the more general construction can be formalised in a predicative topos, as we will now explain.

We have decided to follow \cite{moerdijkpalmgren02} in the formalisation of the notion of a site. So here the basic categorical structure of an internal site consists of an internal category \ct{C} together with a collection of covering families Cov$(C)$ for every object $C$ of \ct{C}. This is formalised by a commutative square of the form
\begin{equation} \label{site}
\begin{gathered}
\xymatrix{ *+\txt{ \underline{Cov} } \ar[r]^m \ar[d]_{\phi} & C_1 \ar[d]^{\txt{cod}} \\
*+\txt{Cov} \ar[r]_n & C_0, }
\end{gathered}
\end{equation}
where $C_1$ is the object of arrows and $C_0$ the object of objects of the internal category \ct{C}, while cod is the codomain map. So any $U \in$ Cov$(C)$ gives rise to an indexing set \underline{Cov}$_U$, indexing a family of arrows all with codomain $C$. Such a covering family $U$ will therefore typically be denoted by $( \alpha_i: C_i \to C \, | \, i \in I )$, where $I$ is the indexing set. (Here it is to be understood that $\alpha_i$ can be equal to $\alpha_j$ for different $i$ and $j$.)  

For a \emph{site}, the following axiom should hold in the internal logic:
\begin{description}
\item[(C)] For any covering family $( \alpha_i: C_i \to C \, | \, i \in I )$ of $C$ and any arrow $f: D \to C$, there exists a covering family $( \beta_j: D_j \to D \, | \, j \in J )$ such that every composite $f\beta_j$ factors through some $\alpha_i$.
\end{description}

Sites are the right context to formulate the notion of a \emph{sheaf}. An (internal) \emph{presheaf} \psh{P} on an internal category \ct{C} consists of an indexed family $\{ \psh{P}(C) \}_{C \in C_0}$ (given by a map $\psh{P} \to C_0$), together with an action by $C_1$. The idea is that any $p \in \psh{P}(C)$ is acted on by a morphism $f: D \to C$ in $C_1$ to determine an element $p \cdot f \in \psh{P}(D)$, where this action is subject to the following two equations (whenever they make sense):
\begin{eqnarray*}
p \cdot \id & = & p, \\
(p \cdot f) \cdot g & = & p \cdot (fg).
\end{eqnarray*}
Now assume \ct{C} is the underlying category of an internal site $(\ct{C}, {\rm Cov})$, and let \psh{P} be a presheaf on \ct{C}. A \emph{compatible family} on an object $C \in C_0$ consists of a covering family $( \alpha_i: C_i \to C \, | \, i \in I )$ of $C$, together with for every $i \in I$ an element $p_i \in \psh{P}(C_i)$, in such a way that for any pair $i, j \in I$ and any pair of maps $f, g \in C_1$ for which we have $\alpha_if = \alpha_jg$ the equation $p_i \cdot f = p_j \cdot g$ is satisfied. An element $p \in \psh{P}(C)$ will be called the \emph{amalgamation} of this compatible family, in case $p \cdot \alpha_i = p_i$ for every $i \in I$. The presheaf \psh{P} will be called a \emph{sheaf} when every compatible family has a \emph{unique} amalgamation.

The next definition formulates different notions of a site.
\begin{defi}{sites}
Let $(\ct{C}, {\rm Cov})$ be a site.
\begin{enumerate}
\item It will be called a \emph{Grothendieck site} when it satisfies the following closure conditions:
\begin{description}
\item[(M)] For any object $C$, there is a covering family $U \in \mbox{Cov}(C)$ such that $(\id_C: C \to C) \in U$.
\item[(L)] Whenever there are a covering family $( \alpha_i: C_i \to C \, | \, i \in I )$ of $C$ and families $(\beta_{ij}: \xymatrix{ C_{ij} \ar[r] & C_{i}} \, | \, j \in I_i)$ covering $C_i$ for every $i \in I$, there is a family $(\gamma_l: D_l \to C \, | \, l \in L)$ such that every $\gamma_l$ factors through some $\alpha_i\beta_{ij}$.
\end{description}
\item It will be called a \emph{(strong) collection site}, if (\ref{site}) is a (strong) collection square.
\end{enumerate}
\end{defi}

As Johnstone explains in \cite{johnstone02a}, in topos theory (M) and (L) are closure conditions that ``might just as well be there'', but are not essential to the notion of a site. This is backed up by the result that in a topos there is for every internal site a Grothendieck site leading to an equivalent category of sheaves (if that happens, we say that the sites are \emph{equivalent}). The idea behind the construction of the Grothendieck site is to generate it inductively by closing off under the conditions (M) and (L). In a predicative topos, we do have W-types and the inductive definition does not present any problems; we do have some choice issues, however, for which we do seem to need that we start with a collection site. It is fortunate, then, that we can show (compare Lemma 8.3 in \cite{moerdijkpalmgren02}):

\begin{lemm}{equivcollsite} Let \ct{E} be a (strong) predicative topos. Then for every internal site \ct{C} in \ct{E} there exists an equivalent (strong) collection site.
\end{lemm} 
\begin{proof}
If we have an internal site in \ct{E} represented by
\begin{displaymath}
\begin{gathered}
\xymatrix{ *+\txt{ \underline{Cov} } \ar[r]^m \ar[d]_{\phi} & C_1 \ar[d]^{\txt{cod}} \\
*+\txt{Cov} \ar[r]_n & C_0 }
\end{gathered}
\end{displaymath}
and \ct{E} satisfies (strong) AMC, we can apply it to the map $\phi$ to obtain
\begin{displaymath}
\begin{gathered} 
\xymatrix{ E \ar@{->>}[r] \ar[d] & *+\txt{ \underline{Cov} } \ar[r]^m \ar[d]_{\phi} & C_1 \ar[d]^{\txt{cod}} \\ D \ar@{->>}[r] &
*+\txt{Cov} \ar[r]_n & C_0, }
\end{gathered}
\end{displaymath}
where the left square is both covering and a (strong) collection square. The idea is to replace the original site by the one represented by the outer rectangle. It is easy to see that this defines an equivalent site which is also (strong) collection.
\end{proof}

\begin{prop}{eqGrsites}
Let \ct{E} be a (strong) predicative topos. For every (strong) collection site $(\ct{C}, {\rm Cov})$ in \ct{E}  there exists an equivalent Grothendieck (strong) collection site $(\ct{C}, {\rm COV})$. 
\end{prop}
\begin{proof} The idea is to build COV as the object inductively generated by the following inference rules:
\begin{displaymath}
\begin{array}{ccc}
\infer{(\id_C: C \to C) \in {\rm COV}(C)}{} & &
\infer{(\alpha_i \circ \beta \, | \, \beta \in V_i) \in {\rm COV}(C).}{( \alpha_i: C_i \to C \, | \, i \in I ) \in {\rm Cov}(C) & V_i \in {\rm COV}(C_i)}
\end{array}
\end{displaymath}
To show that one can achieve this, define an endofunctor 
$F: \ct{E}/C_0 \to \ct{E}/C_0$ as follows ($C \in C_0$):
\[ (FX)_C = 1 + \sum_{U \in \mbox{Cov}(C)} \prod_{i \in \underline{\mbox{Cov}}_U} X_{\mbox{dom}(m(i))}. \]
This is an example of what Gambino and Hyland in \cite{gambinohyland04} call a \emph{dependent polynomial functor}. They show that, in the presence of W-types, these have initial algebras, so we may define the new object of covering families COV as the initial algebra for $F$; as it is a fixed point, it will satisfy ($C \in C_0$):
\[ \mbox{COV}(C) = 1 + \sum_{U \in \mbox{Cov}(C)} \prod_{i \in \underline{\mbox{Cov}}_U} \mbox{COV}(\mbox{dom}(m(i))). \]
In addition, initial algebras allow definition by recursion on their elements and this we will use to define the new object \underline{COV} over COV and the new arrow $M$, thereby completing the definition of the site $(\ct{C}, {\rm COV})$:
\begin{equation*} 
\begin{gathered}
\xymatrix{ *+\txt{ \underline{COV} } \ar[r]^M \ar[d] & C_1 \ar[d]^{\txt{cod}} \\
*+\txt{COV} \ar[r] & C_0. }
\end{gathered}
\end{equation*}
Elements $V$ in COV (over $C \in C_0$) are either * (the unique element of 1) or of the form sup$_U(t)$, where $U \in \mbox{Cov}(C)$ and $t: \underline{\mbox{Cov}}_U \to \mbox{COV}$. So we may define \underline{COV}$_V$ by recursion on $V$ as:
\begin{eqnarray*}
\mbox{\underline{COV}}_* & = & 1, \\
\mbox{\underline{COV}}_{\mbox{sup}_U(t)} & = & \sum_{i \in \underline{\mbox{Cov}}_U} \mbox{\underline{COV}}_{t(i).}
\end{eqnarray*}
The definition of $M$ runs as follows. The unique element of \underline{COV}$_*$ is sent to $1_C: C \to C$. An element in \underline{COV}$_{\mbox{sup}_U(t)}$ is of the form $(i, k)$, with $i \in \underline{\mbox{Cov}}_U$ and $k \in \mbox{\underline{COV}}_{t(i)}$. Such an element $(i, k)$ is sent to $m(i) \circ M(k)$. 

We will now outline why the constructed objects have the desired properties. First, we need to show that $(\ct{C}, {\rm COV})$ is a Grothendieck site. It is obvious that COV now satisfies the rule (M). By induction on the construction of the covering family $V \in {\rm COV}(C)$, one can show that the covering families in COV are closed under (C) and (L), using that $(\ct{C}, {\rm Cov})$ is a collection site. We will give the argument for (C), the proof of (L) being similar. If $U = (\id_C: C \to C)$, then $U$ clearly satisfies (C), so it remains to consider the case where $U$ is obtained by the second inference rule. So let $V = {\rm Cov}(C)$ be of the form $( \alpha_i: C_i \to C \, | \, i \in I )$ and assume that the covering families $V_i \in {\rm COV}(C_i)$ are closed under (C). Now let $f: D \to C$ be any map in the internal category \ct{C}. I need to show that there is an $U' \in {\rm COV}(D)$ such that for any $g \in U'$, the arrow $f g$ factors through some $\alpha_i \beta$ with $\beta \in V_i$. As $(\ct{C}, {\rm Cov})$ is a site, there is a covering family $(\delta_j: D_j \to D \, | \, j \in J) \in {\rm Cov}(D)$ such that every $f \delta_j$ factors through some $\alpha_i$. Using that the $V_i$ have the property (C), this means that the following holds in the internal logic of \ct{E}:
\begin{quote}
For all $j \in J$, there is an $i \in I$, a morphism $h: D_j \to C_i$ and a covering family $V' \in {\rm COV}(D_j)$ such that: (1) $f \delta_j = \alpha_i h$, and (2) every $h \beta'$ with $\beta' \in V'$ factors through some $\beta$ with $\beta \in V_i$.
\end{quote}
Here we use that $(\ct{C}, {\rm Cov})$ is a collection site: so there is a covering family $(\gamma_k: D_k \to D \, | \, k \in K) \in {\rm Cov}(D)$ together with an map $p: K \to J$ satisfying $\delta_{pk} = \gamma_k$ for all $k \in K$, such that appropriate $i$, $h$ and $V'$ are given as a function of $k \in K$: call these $i_k$, $h_k$ and $V'_k$. One can now build the desired $U' \in {\rm COV}(D)$ by applying the second inference rule to $(\gamma_k: D_k \to D \, | \, k \in K) \in {\rm Cov}(D)$ and $V'_k \in {\rm COV}(D_k)$. Then every $g \in U'$ is of the form $\gamma_k \beta'$ with $\beta' \in V'_k$, and hence $f \gamma_k \beta' = \alpha_{i_{pk}} h_k \beta'$ factors through $\alpha_{i_{pk}} \beta$ for some $\beta \in V_{i_{pk}}$. This proves that $(\ct{C}, {\rm COV})$ is a site. As said, the argument that it satisfies (L), and is therefore a Grothendieck site, is very similar.

The proof that $(\ct{C}, {\rm COV})$ is a (strong) collection site also goes along very similar lines. We give the construction, leaving verifications to the reader. Let $V = (\gamma_j: C_j \to C \, | \, j \in J) \in {\rm COV}(C)$, and assume:
\[ \forall j \in J \, \exists x \in X \, \varphi(j, x). \]
We need to show that there exists another covering family $V' = (\delta_k: C_k \to C \, | \, k \in K)$ of $V$ such that $x$ can be given as a function of $k$. The argument proceeds by induction on the construction of $V$. The statement is trivial when $V = (\id_C: C \to C)$, so assume that $V$ has been built by the second inference rule, meaning that there are $( \alpha_i: C_i \to C \, | \, i \in I) \in {\rm Cov}(C)$ and $( \beta_{j}: C_j \to C_i \, | \, j \in J_i ) \in {\rm COV}(C_i)$ such that
\[ V = \{ \alpha_i \beta_j \, | \, i \in I, j \in J_i \}. \]
By induction hypothesis, the following statement is true in the internal logic of \ct{E}:
\begin{quote}
For all $i \in I$, there exists an element $V = (\epsilon_k: C_k \to C_i \, | \, k \in K) \in {\rm COV}(C_i)$ together with a map (cover) $p: K \to J_i$ and a function $g: K \to X$ such that for all $k \in K$: (1) $\beta_{pk} = \epsilon_{k}$ and (2) $\varphi((i, pk), gk)$.
\end{quote}
Since $(\ct{C}, {\rm Cov})$ is a (strong) collection site, there is a covering family $(\eta_l : C_l \to C \, | \, l \in L ) \in {\rm Cov}(C)$ and a map (cover) $r: L \to I$ such that $\alpha_{rl} = \eta_l$ and appropriate $V$, $p$ and $g$ are given as a function of $l \in L$. To obtain the right covering family, we apply the second inference rule to $(\eta_l : C_l \to C \, | \, l \in L )$ and the $V_l$. 

It remains to check that the sites $(\ct{C}, {\rm Cov})$ and $(\ct{C}, {\rm COV})$ are equivalent. This is comparatively easy: if \psh{P} is a sheaf with respect to Cov, one shows that it satisfies the sheaf condition for $V \in {\rm COV}(C)$ by a straightforward induction on the generation of $V$. Conversely, every $(\alpha_i : C_i \to C \, | \, i \in I) \in {\rm Cov}(C)$ also occurs as an element of COV($C$) (use the second inference rule with $V_i$ consisting solely of the arrow $\id_{C_i}$), so anything that is a sheaf with respect to COV is certainly also a sheaf with respect to Cov.
\end{proof}

To summarise:

\begin{theorem} \label{eqGrcollsite}
Let \ct{E} be a (strong) predicative topos. Then every site in \ct{E} is equivalent to a Grothendieck (strong) collection site.
\end{theorem}

This we can use to establish our final closure property of predicative toposes:

\begin{theo}{sheavesorpredtoposes}
(Strong) predicative toposes are closed under internal sheaves.
\end{theo}
\begin{proof}
It is well-known that the category of internal presheaves has finite limits and is locally cartesian closed (see \cite{moerdijkpalmgren00}, for example). The same applies to internal sheaves, because finite limits and exponentials in sheaves are computed as in presheaves.

To show that the internal sheaves form a pretopos, we first observe that, in view of the previous theorem, we may assume that we are taking sheaves over an internal Grothendieck collection site. In that case, \cite[Lemma 8.1]{moerdijkpalmgren02} tells us that a sheafification functor, a left adjoint to the inclusion of presheaves in sheaves, exists (that argument also works for collection sites that are not strong). Hence sums and quotients in sheaves can be constructed by sheafifying sums and quotients in presheaves.

Furthermore, to show that the category of sheaves has W-types, we use AMC as in \cite[Theorem 4.21]{bergmoerdijk10b}. (It is clear from the proof that ordinary AMC suffices.)

Finally, that AMC is inherited by sheaves is shown in detail in \cite[Section 5]{bergmoerdijk10}; for strong AMC that was shown in \cite[Section 10]{moerdijkpalmgren02}. In fact, the latter proof can be simplified considerably by using the equivalence of SAMC with RP, but we will not go into the details here.
\end{proof}

\section{Conclusion and open questions}

As we said in the introduction, we hope that this paper can form a starting point for future work on predicative topos theory. We will finish it by indicating some directions for future research.

\begin{enumerate}
\item We have established that predicative toposes are closed under internal realizability and sheaves. But there are more closure properties one could have a look at, such as glueing, forming coalgebras for a cartesian comonad and taking filter quotients. Some results in this direction have been obtained for $\Pi W$-pretoposes (see \cite{moerdijkpalmgren00,berg09}). The question is whether these results still go through now that we have added AMC.
\item In view of the results in Section 5, it would be interesting to see whether algebraic theories have free algebras internally in a predicative topos. (This question was originally posed to us by Alex Simpson.) We expect that they do. More generally, it would be an interesting project to investigate which inductively defined structures and which initial algebras exist in a predicative topos.
\item One would like to have more examples of predicative toposes. Possible candidates are the ex/reg-completions of various quasitoposes, where one would hope to be able to further exploit \reftheo{exregcomplforpredtoposes}; in our opinion, this would be especially interesting for the modified assemblies (see \cite{vanoosten08}). One could also try to put the last couple of examples from Section 6 in the context of the theory of typed pcas (see \cite{longley99} and \cite{lietzstreicher02}), or in the framework of \cite{birkedal02}. In this way one should also be able to find new examples of predicative toposes that are not toposes.
\item Given the recent interest in homotopy type theory, it is natural to ask what is the connection to the notion of a predicative topos. For example, do the hSets (the types of h-level two) in homotopy type theory form a predicative topos? (This question is due to Bas Spitters.) 
\end{enumerate}

\bibliographystyle{plain} \bibliography{ast}

\end{document}